\documentclass[12pt]{amsart}

\usepackage{amsmath,amssymb,latexsym}
\usepackage{amsfonts,amstext,amsthm,amscd}
\usepackage{graphicx, color}
\usepackage{enumerate}
\usepackage{amsthm}
\theoremstyle{plain}

\usepackage{color}

\newtheorem{theorem}{Theorem}[section]

\newtheorem{proposition}[theorem]{Proposition}
\newtheorem{definition}[theorem]{Definition}
\newtheorem{lemma}[theorem]{Lemma}

\newtheorem{corollary}[theorem]{Corollary}
\newtheorem{remark}[theorem]{Remark}

\newcommand{\A}{\mathbb A}
\newcommand{\Af}{\A_f}

\newcommand{\dAf}{d_{\mathbb{A}_f}}

\newcommand{\C}{\mathbb C}

\newcommand{\Co}{\mathrm C}
\newcommand{\D}{\mathcal D}
\newcommand{\Da}{D^{\alpha}}
\newcommand{\Db}{D^{\beta}}
\newcommand{\Dab}{D^{\alpha,\beta}}
\newcommand{\ma}{m^{\alpha}}
\newcommand{\mb}{m^{\beta}}
\newcommand{\mab}{m^{\alpha,\beta}}

\newcommand{\Dom}{\mathrm {Dom}}
\newcommand{\B}{\mathcal B}

\newcommand{\F}{\mathcal F}

\newcommand{\N}{\mathbb N}

\newcommand{\Q}{\mathbb Q}

\newcommand{\R}{\mathbb R}

\newcommand{\Z}{\mathbb Z}

\newcommand{\Zz}{\widehat{\Z}}

\newcommand{\abs}[1]{\left\vert#1\right\vert}

\newcommand{\norm}[1]{\left\Vert#1\right\Vert}
\newcommand{\ord}{\mathrm{ord}}
\newcommand{\To}{\longrightarrow}

\date{\today}

\begin{document}

\title[Invariant Parabolic   equations and Markov process Ad\`eles ]{Invariant Parabolic   equations and Markov process on  Ad\`eles 
}
\author[V.A. Aguilar--Arteaga and S. Estala--Arias]{ Victor A. Aguilar--Arteaga$^*$   and Samuel Estala--Arias$^{**}$}

\address{$^{*}$ Departamento de Matem\'aticas, CINVESTAV, Unidad Quer\'etaro, M\'exico} 
\email{aguilarav@math.cinvestav.mx}

\address{$^{**}$ Instituto de Matem\'aticas, UNAM.} 
\email{samuelea82@gmail.com}



\begin{abstract}In this article a class of additive invariant positive selfadjoint pseudodifferential unbounded operators  on 
$L^2(\Af)$,  where  $\Af$ is the ring of finite ad\'eles of the rational numbers, is considered   to state a Cauchy problem of
parabolic--type equations. These operators  come from a set of additive invariant non-Archimedean metrics on  $\Af$. The fundamental 
solutions of these parabolic equations  determines normal transition functions of  Markov process on $\Af$. Using the fractional 
Laplacian on the Archimedean place, $\R$, a class of parabolic--type equations on the complete ad\`ele ring, $\A$, is obtained.
\end{abstract}

\maketitle

\section*{Introduction}
\label{introduction}

 The theory of stochastic processes on locally compact  groups and their relationship to  
pseudodifferential equations have become an intense subject of study from some decades ago. In particular, the research of 
pseudodifferential operators over non--Archimedean spaces, such as the field $\Q_p$ of $p$--adic numbers or more general local fields,
has been deeply explored by several authors  (see e.g., \cite{ACE1,AKS, KK2, Koc, KKS,VVZ, Zun} and the references therein). Pursuing a 
generalisation, there have
been many attempts to formalise on both aspects of the theory and several general frameworks have been proposed.

 This article deals with the study of a certain class of parabolic-type  pseudodifferential equations and its associated Markov 
 stochastic processes on the complete ad\`ele group $\A$ of the rational numbers $\Q$. This locally compact topological ring can be
 factorised as 
$\A=\R\times \Af$, where $\Af$ is the totally disconnected part of $\A$ and $\R$ its connected component at the identity. According
to a theorem of Struble, $\Af$ and $\A$ are metrizable with an additive invariant metric all whose spheres are bounded 
(see  \cite{Str}, including \cite{McF, Gro}). In \cite{TZ}, an  ultrametric on $\Af$  is the departing point  to study a 
parabolic--type equation on $\Af$ and its extension to  $\A$, considering the fractional Laplacian on $\R$. More recently, a class of 
invariant ultrametrics on the finite ad\'ele ring $\Af$
has been introduced (\cite{CE}). These ultrametrics are the only additive invariant
metrics whose balls centred at zero  are compact and open subgroups of $\Af$.  If the   radius of any ball agree with its Haar
measure,   the ultrametric is   called \emph{regular}. These ultrametrics have a subclass here called  symmetric. Any symmetric 
regular ultrametric is given by a strictly increasing sequence of natural numbers, beginning with one, totally order by division and 
cofinal with the natural numbers. In this article we  
study a class of parabolic--type equation and its associated Markov processes of $\Af$ related to the afore mentioned symmetric 
regular ultrametric. Similar to   \cite{TZ}, considering  the Archimedean fractional Laplacian on $\R$, we extend the mentioned 
results to the ring of ad\`eles.

Given a symmetric  regular ultrametric $\dAf$ on $\Af$ and a real number $\alpha >0$ leads naturally to  a positive selfadjoint
pseudodifferential operator $\Da$ on $L^2(\Af)$ and to  study the Cauchy problem of a parabolic--type equation  on $\Af$.  
Nonetheless, the classical techniques of Fourier Analysis and a geometrical point of view coming from the regular condition on the
ultrametrics give us a bound that allows to prove regular considerations on the Markov process.  

%

Given $0<\beta \leq 2$ and the fractional Laplacian, $\Db$, in $L^2(\R)$, a positive selfadjoint pseudodifferential operator 
$\Dab=\Da+\Db$ on $L^2(\A)$,  is defined. By construction, any of  these operators are invariant under translations. The following 
theorem encloses  the results of  this writing. 

 \textbf{Theorem \ref{solution_heatequation}:} If $f$ is any complex valued square integrable function on   $\Dom(\Dab)$, then the Cauchy problem

\begin{equation*}
\begin{cases}
\frac{{\partial}u(x,t)}{{\partial t}} + \Dab u(x,t) = 0,  \ x \in \A, \ t > 0,  \\  u(x,t) = f(x)
\end{cases}
\end{equation*} 
has a  solution $u(x,t)$ determined by the convolution of $f$ with the
heat kernel $Z(x,t)$. Moreover, $Z(x,t)$ is the transition density  of a time and space homogeneous Markov process which is bounded, right--continuous and has no discontinuities other than jumps.

The exposition is organised as follows.  Section \ref{FiniteAdeleGroup}
presents  preliminary results on non--Arquimedean and Harmonic analysis on the finite ad\`ele group. In section \ref{ParabolicEquationsonAf} 
the pseudodifferential operator $\Da$ defined on $L^2(\Af)$ is introduced, and the corresponding Cauchy problem related to the homogeneous  heat  equations is treated. Finally, in Section \ref{ParabolicEquationsonA}, a Cauchy problem for the pseudodifferential operator $\Dab$ is studied.

\section[The finite ad\`ele ring of $\Q$]{The finite ad\`ele ring of $\Q$}
\label{FiniteAdeleGroup}

This section introduces the ring of finite ad\`eles, $\Af$, as  a completion of the rational numbers with respect to    additive invariant ultrametrics. For a detailed description of these results we quote \cite{ACE2} and \cite{CE}. Complete information on the classical definition of the   ad\`ele rings can be found  in \cite{Wei}, \cite{RV}, \cite{Lan}, \cite{GGPH}.

\subsubsection{ Ultrametrics on $\Af$} Let $\N=\{1,2,\cdots\}$ denote the set of natural numbers. Let  $\psi(n)$ be implicitly 
defined by any strictly increasing sequence of natural numbers $\{e^{\psi(n)}\}_{n=0}^{\infty}$, which is totally ordered by division
and cofinal with the natural numbers, and with $e^{\psi(0)}=1$\footnote{The notation chosen here is inspired by the second Chebyshev 
function but it should not be confused with it, see \cite{Apo}.}. The function $\psi(n)=\log(e^{\psi(n)})$  can be defined to any integer number $n$, as 
\begin{equation*}
 \psi(n)=\begin{cases} 
 \frac{n}{\abs{n}} \psi(\abs{n}) &  \text{ if } n \neq 0, \\
 0  & \text{ if } n=0. 
 \end{cases}
\end{equation*}
For any integer $n$ define $\Lambda(n)$ implicitly by the equation
$$
e^{\Lambda(n)}=\frac{e^{\psi(n)}}{e^{\psi(n-1)}}.
$$
Then, for any integers $n > m$, there is a relation
$$
\psi(n) - \psi(m) = \sum_{k=m+1}^{ n} \Lambda(k) \quad 
\text{ or equivalently } \quad e^{\psi(n)}/e^{\psi(m)} 
= \prod_{k =m+1}^n e^{\Lambda(k)}. 
$$ 
The collection $\{e^{\psi(n)}\Z \subset \Q\}_{n\in \Z}$ is a neighbourhood base of zero for an additive invariant  topology on $\Q$ formed by  open and closed subgroups. Additionally, it satisfies the properties:
$$
\bigcap_{n\in \Z} e^{\psi(n)}\Z  =0, \quad \text{ and } \quad   \bigcup_{n\in \Z} e^{\psi(n)}\Z=\Q.
$$
For any element $x\in \Q$  the \textsf{adelic order} of $x$ is given by:
\begin{equation*}
 \ord(x):=\begin{cases} 
 \max\{n:x\in e^{\psi(n)}\Z \} &  \text{ if } x \neq 0, \\
  \infty  & \text{ if } x=0. 
 \end{cases}
\end{equation*}
This function satisfies the following properties:
\begin{enumerate}
 \item[$\bullet$] $\ord_{\Af}(x)\in \Z\cup \{\infty \}$ and $\ord_{\Af}(x)=\infty$ if and only if $x=0$.  
\item[$\bullet$] $\ord_{\Af}(x+y)\geq \min\{\, \ord_{\Af}(x), \ord_{\Af}(y)\, \} $.
\end{enumerate}

\textsf{The ring of finite ad\`eles} $\Af$ is defined as the completion of $\Q$ with respect to the non--Archimedean ultrametric, $d:\Q\times \Q\longrightarrow \R^+\cup \{ 0\}
$, given by \footnote{This definition does not depend on the initial filtration.}
$$
d(x,y)=e^{-\psi(\ord(x-y))}.
$$
Every non--zero finite ad\`ele $x\in \Af$ has  a unique  series representation of the form
\[ x = \sum_{l=\gamma}^\infty x_l e^{\psi(l)}, \quad ( x_l \in \{0,1, \ldots, e^{\Lambda(l+1)}-1\}) \]
with $x_\gamma \neq 0$ and $ \gamma=\ord(x) \in \Z$. This series is convergent in the ultrametric of $\Af$ and the numbers $x_l$ appearing in the representation of $x$ are unique. The value $\gamma$, with $\gamma(0)=+\infty$ is  \textsf{the finite adelic order of $x$}. 

The \emph{\textsf{fractional part}} of a finite ad\`ele $x\in \Af$ is defined by
\begin{equation*}
\{x\} := 
\begin{cases} 
\displaystyle{\sum_{k=\gamma(x)}^{-1} a_k  e^{\psi(k)}}	& \text{ if } \gamma(x) < 0, \\
0 																		& \text{ if } \gamma(x) \geq 0. 
\end{cases}
\end{equation*}

The ultrametric $d(x,y)$ makes the  ring  $\Af$  a second countable locally compact totally disconnected topological ring. The \textsf{ring of } \textsf{adelic integers} is the  unit ball 
\[ \Zz = \{\, x \in  \Af \, : \, \norm{x} \leq 1 \,\} \] 
which is the maximal compact and open subring of $\Af$ as well. Denote by $dx$ the Haar measure of the topological Abelian group $(\Af,+)$ normalised such that the Haar measure of  $\Zz$ is equal to one.

Note that the ultrametric $d$ takes values in the set $\{ e^{\psi(n)} \}_{n\in \Z}\cup \{ 0\}$ and the balls centred at zero $B_{n}$ are the sets 
$\{e^{\psi(n)}\Zz \subset \Af\}_{n\in \Z}$, that is, 
\[ B_{n}:=B(0,e^{\psi(n)})=e^{-\psi(n)}\Zz. \]
We denote the \textsf{sphere} centred at zero and radius $e^{\psi(n)}$ as $S_{n}$, i.e.
$$
S_{n}:=S(0,e^{\psi(n)})=B_{n}\backslash B_{n-1}.
$$

The norm induced by this ultrametric is given by 
$$
\norm{x}=e^{- \psi (\ord(x))}
$$ and $\norm{x}=e^{\psi(n)}$ if and only if $x \in S_n$.

A function $\phi:\Af \To \C$ is  \textsf{locally constant} if for any $x\in \Af$, there exists an integer $\ell(x)\in \Z$ such that
\[ \phi(y) = \phi(x), \quad \text{for all}\; y \in B_{\ell(x)}(x) , \]
where $B_{\ell(x)}(x)$ is the closed ball with centre at $x$ and radius $e^{\psi(\ell(x))}$.
Let $\D(\Af)$ denote the  $\C$--vector space  of all locally constant functions with compact support on $\Af$. The vector space $\D(\Af)$ is called  \textsf{Bruhat--Schwartz space} of $\Af$ and an element 
$\phi\in \D(\Af)$  a \textsf{Bruhat--Schwartz function} (or simply a \textsf{test function}) on $\Af$.

If $\phi$ belongs to $\D(\Af)$ and $\phi(x)\neq 0$  for some $x\in \Af$, there exists a largest $\ell=\ell(\phi)\in \Z$, which is called \textsf{the parameter of constancy} of $\phi$, such that, for any $x\in \Af$, we have 
\[ \phi(x+y)=\phi(x), \text{ for all } y \in B_{\ell}. \]

Denote by $\D_k^{\ell}(\Af)$ the finite dimensional vector space consisting of functions whose   parameter of constancy is greater than or equal to $\ell$ and whose  support is contained in $B_k$. 

A sequence $(f_m)_{m\geq 1}$ in $\D(\Af)$ is a Cauchy sequence if there exist $k,\ell \in \Z$ and $M>0$ such that 
$f_m\in \D_k^\ell(\Af)$ if $m\geq M$ and $(f_m)_{m\geq M}$ is a Cauchy sequence in $\D_k^\ell(\Af)$. That is, 
$$  
\D(\Af) =\varinjlim_{\ell \leq k} \D_k^{\ell}(\Af).
$$
With this topology the space $\D(\Af)$ is a complete locally convex topological vector over $\C$. It is also a nuclear space because
it  is the  inductive  limit of the countable family of finite dimensional vector spaces  $\{\D_k^{\ell}(\Af)\}$.

For each compact set $K\subset \Af$, let $\D(K)\subset \D(\Af)$ be the subspace of test functions whose support is contained in $K$. The space $\D(K)$ is dense in $\Co(K)$, the space of complex--valued continuous functions on $K$.

An additive character of the field $\Af$ is defined as a continuous function $\chi:\Af \To \C$ such that $\chi(x+y)=\chi(x)\chi(y)$ and $\abs{\chi(x)}=1$, for all $x,y \in \Af$.  The function $\chi(x)=\exp(2\pi i\{x\})$ defines a canonical additive character of $\Af$, which is trivial on $\Zz$ and not trivial outside $\Zz$, and all characters of $\Af$ are given by $\chi_{\xi}(x)=\chi(\xi x)$, for some $\xi \in \Af$. The Fourier transform of a test function $\phi \in \D(\Af)$ is given by the formula
$$
\F\phi(\xi) = \widehat{\phi}(\xi) = 
\int_{\Af} \phi(x) \chi(\xi x)dx, \qquad  (\xi \in \Af). 
$$ 
The Fourier transform is a continuous linear isomorphism of the space $\D(\Af)$ onto itself and the following inversion formula holds:
\[ \phi(x) = \int_{\Af} \widehat{\phi}(\xi) \chi(-x\xi) d\xi \qquad \big(\phi\in \D(\Af)\big). \]
Additionally, the Parseval -- Steklov equality reads as: 
$$
\int_{\Af} \phi(x) \overline{\psi(x)} dx =\int_{\Af} \widehat{\phi}(\xi) \overline{\widehat{\psi}(\xi)} d\xi,
\qquad \big(\phi,\psi \in \D(\Af)\big).
$$
Last but not least, the Hilbert space $L^2(\Af)$  is a separable Hilbert space and the extended Fourier transform $\F:L^2(\Af) \To L^2(\Af)$ is an isometry of Hilbert spaces. Moreover,  the Fourier inversion formula and  the Parseval -- Steklov identity hold on $L^2(\Af)$.

\begin{remark} \label{oscillatory_integral}
The Haar measure of any ball is equal to its radius: 
$$
\int_{y+B_{n}}d\xi = \int_{B_{n}}d\xi  =  \int_{e^{-\psi(n)}\Zz} d\xi =  e^{\psi(n)} \qquad ( y \in \Af, \, n \in \Z), 
$$
and the area of a sphere is given by
$$
\int_{y+S_{n}}d\xi = \int_{S_{n}}d\xi  =  e^{\psi(n)}-e^{\psi(n-1)} \qquad ( y \in \Af, \, n \in \Z).
$$
Moreover, for any $n \in \Z$  the following formulae hold: 
\begin{enumerate}
 \item 
 \begin{equation*}
 \int_{B_{n}}\chi(-\xi x)dx=\begin{cases} 
 e^{\psi(n)} &  \text{ if } \|\xi \| \leq e^{-\psi(n)}, \\
  0  & \text{ if } \|\xi \|> e^{-\psi(n)}. 
 \end{cases}
\end{equation*}

\item 
 \begin{equation*}
 \int_{S_{n}}\chi(-\xi x)dx=\begin{cases} 
 e^{\psi(n)}-e^{\psi(n-1)} &  \text{ if } \|\xi \| \leq e^{-\psi(n)}, \\
-e^{\psi(n-1)} & \text{ if } \| \xi \|=e^{-\psi(n-1)}, \\
  0  & \text{ if } \|\xi \|\geq e^{-\psi(n-2)}. 
 \end{cases}
\end{equation*}

\end{enumerate}

\end{remark}

\section[Parabolic-type   equations on $\Af$]{Parabolic-type  equations on $\Af$}
\label{ParabolicEquationsonAf}

This section introduces a positive selfadjoint pseudodifferential unbounded operator $\Da$ on $L^2(\Af)$, the Hilbert space of square
integrable functions on $\Af$, and solves the abstract Cauchy problem for  the homogeneous heat  equation  on $L^2(\Af)$ related to
$\Da$. The properties of general evolution equations on Banach spaces can be found in \cite{EN}, \cite{CH} and \cite{Paz}.  The reader can 
consult these and more topics in the excellent books \cite{Igu}, \cite{AKS}, \cite{VVZ} and \cite{Zun}.

\subsection[Pseudodifferential operators on $\Af$]{Pseudodifferential operators on $\Af$}
\label{operators}

For any $\alpha > 0$, consider the   pseudodifferential operator 
$\Da:\Dom(\Da) \subset L^2(\Af)\To L^2(\Af)$ defined by the formula
$$ \Da \phi(x)= \F^{-1}_{\xi \to x}[\norm{\xi}^\alpha \F_{x \to \xi}[f]], $$
for any $\phi$ in the domain
\begin{equation*}
\Dom(\Da):=\left\{\, f \in L^2(\Af) : \, \norm{\xi}^\alpha\widehat{f}(\xi) \in L^2(\Af)  \,\right \}, 
\end{equation*}
This operator is  a  \textsf{pseudodifferential operator with symbol} $\norm{\xi}^\alpha$.   It can be seen that the unbounded 
operator $\Da$, with domain $\Dom(\Da)$, is a positive selfadjoint operator which is  diagonalized by the (unitary) Fourier
transform.  In other words, the following diagram commutes:
\begin{equation}\label{Conmutatividad(Da-ma)}
\begin{CD}
L^{2}(\Af) @>\mathcal{F}>>L^{2}(\Af) \\
@V{ \Da  }VV @VV{ m^\alpha  }V\\
L^{2}(\Af) @>\mathcal{F}>> L^{2}(\Af), 
\end{CD}
\end{equation}
where $m^{\alpha}: L^2(\Af) \To L^2(\Af)$ is  the  multiplicative operator given by $f(\xi) \longmapsto \norm{\xi}^\alpha f(\xi)$,  with (dense) domain 
\begin{equation*}
\Dom(m^{\alpha}):=\left\{\, f \in L^2(\Af) \,: \,  \norm{\xi}^\alpha f(\xi) \in L^2(\Af)  \,\right \}. 
\end{equation*}

As a result, several  properties of $\Da$, depending only on the inner product of $L^{2}(\Af)$, can be translated into analogue properties of  the multiplicative operator $m^{\alpha}$. In particular, the characteristic equation  $\Da f = \lambda f$ with  $f \in  L^2(\Af)\setminus \{0\}$ can be solved by applying the Fourier transform. In fact, 
if $\lambda \in \{e^{\alpha \psi(n)}\}_{n\in \Z}$, the characteristic function $1_{S_n}$, of the sphere $S_n$, is a solution of the characteristic equation, $(\norm{\xi}^{\alpha} - \lambda )\widehat{f}(\xi)=0$, of the multiplicative operator $\ma$. Otherwise, if $\lambda \notin \{e^{\alpha \psi(n)}\}_{n\in \Z}$, the function $\norm{\xi}^{\alpha} - \lambda$ is bounded from below and $\lambda$ is in the resolvent set of $\ma$. Since the Fourier transform is unitary, the point spectrum of $\Da$ is the set $\{e^{\alpha \psi(n)}\}_{n\in \Z}$, with corresponding eigenfunctions  
$\{ \F^{-1}(\Delta_{S_n}) \}_{n\in \Z}$. Finally, $\{0\}$ forms part of  the spectrum as a limit point. Each eigenspace is infinite dimensional and there exists a well defined wavelet base which is also made of eigenfunctions.

\begin{remark}
The operator $\Da$ is derived from  the chosen double sequence $\big( e^{\psi(n)}\big)_{n\in\Zz}$.  Any operator $\Da$ is a finite adelic analogue of the Vladimirov operator on $\Q_p$.
 
\end{remark}

\subsection[A Cauchy problem on $L^2(\Af)$]{A Cauchy problem on $L^2(\Af)$}
\label{HeatEquation}
For  $f(x) \in \Dom(\Da) \subset L^2(\Af)$, consider  the abstract Cauchy problem
\begin{equation}
\label{ACPinAf}
\begin{cases}
\frac{{\partial}u(x,t)}{{\partial t}} + \Da u(x,t)=0,  \ x \in \Af, \ t \geq 0   \\
\ u(x,t)=f(x).
\end{cases}
\end{equation} 
This problem  will be pointed as    abstract Cauchy problem  (\ref{ACPinAf}). Notice that for each invariant pseudodifferential operator  $\Da$, the abstract Cauchy problem above  is a finite adelic counterpart of the Archimedean abstract Cauchy problem for the  homogeneous heat equation.

The abstract Cauchy problem (\ref{ACPinAf}) is considered in the sense of the Hilbert space $L^2(\Af)$, that is to say, a function
$u:\Af\times [0,\infty) \To \C$ is called  a \textsf{ solution}   if: 
\begin{enumerate}[a.]
\item $u:[0,\infty) \To L^2(\Af)$ is a continuously differentiable function,
\item $u(x,t)\in \Dom(\Da)$, for all $t\geq 0$ and,
\item $u(x,t)$ is satisfies the initial
value problem (\ref{ACPinAf}). 
\end{enumerate}
 
The abstract Cauchy  problem (\ref{ACPinAf}) is well--posed and  its  solution is well understood from the theory of semigroups of linear operators over Banach spaces.  This solution is described in the following section.

\subsection[Semigroup of operators]{Semigroup of operators}
\label{semigroup_operators}

From  the Hille--Yoshida Theorem, to the positive  selfadjoint operator $\Da$, there corresponds a strongly continuous contraction
semigroup 
$$S(t)=\exp(-t\Da):L^2(\Af)\To L^2(\Af) \qquad(t \geq 0),$$ with infinitesimal generator $-\Da$. It follows that $\{S(t)\}_{t\geq 0}$ has the following properties:

\begin{itemize}

\item For any  $t\geq 0$, $S(t)$ is a bounded operator with  operator norm less or equal to one.
\item The application $t \longmapsto S(t)$ is strongly continuous for $t\geq 0$.
\item  $S(0)$ is the identity operator in $L^2(\Af)$, i.e.  $S(0)(f)=f$, for all $f \in L^2(\Af)$,
\item It has the semigroup property: $S(t)\circ S(s)=S(t+s)$. 
\item If $f\in \Dom(-\Da)$, then  $S(t)f \in \Dom(-\Da)$ for all $t \geq 0$,
the $L^2$ derivative $\frac{d}{dt} S(t)f$ exists, it is continuous for $t \geq 0$, and  is  given by
\[  \frac{d}{dt} S(t)f \Big|_{t=t_0^+} = -\Da S(t_0)f =  -S(t_0)\Da f\qquad  (t_0 \geq 0 ).   \]
\end{itemize}

All this means that $S(t)f$ is a solution of the Cauchy problem (\ref{HeatEquation}) with initial condition
$f\in \Dom(\Da)$.

On the other hand, for  $f(\xi) \in \Dom(\ma) \subset L^2(\Af)$, consider  the abstract Cauchy problem
\begin{equation}
\label{ACPmultiplicativoinAf}
\begin{cases}
\frac{{\partial}u(\xi,t)}{{\partial t}} + \ma u(\xi,t)=0,  \ \xi \in \Af, \ t \geq 0   \\
\ u(\xi,t)=f(\xi).
\end{cases}
\end{equation} 

The solution of this problem is given by the strongly continuous contraction semigroup $
\exp(-tm^\alpha) :L^{2}(\Af) \To L^{2}(\Af)$ given by $$  f(\xi) \mapsto f(\xi) \exp(-t \norm{\xi}^\alpha),
$$  which is the semigroup that corresponds to the positive selfadjoint multiplicative operator $m^\alpha$, under the Hille--Yoshida Theorem and whose infinitesimal generator is equal to $-\ma$. From the fact that  the Fourier transform is an isometry on $L^{2}(\Af)$ and converts the abstract Cauchy problem  (\ref{ACPinAf}) into (\ref{ACPmultiplicativoinAf}), the commutative diagram (\ref{Conmutatividad(Da-ma)}), and corresponding definitions of the infinitesimal generators of $S(t)$ and 
$\exp(-tm^\alpha)$, the following diagram commutes
\begin{equation}
\begin{CD}
L^{2}(\Af) @> \F >> L^{2}(\Af) \\
@V{S(t)}VV @VV\exp(-tm^\alpha)V\\
L^{2}(\Af) @> \F >> L^{2}(\Af). \notag
\end{CD}
\end{equation}

\subsection[A heat kernel]{A heat kernel}
\label{heat_kernel}
In order to describe the theoretical solution given by the Hille--Yosida Theorem we introduce the heat kernel:
\begin{align}
\label{HeatKernel}
Z (x,t) &= \F^{-1}\big(\exp(-t\norm{\xi}^{\alpha})\big) = 
\int_{\Af}\chi(-x\xi)\exp(-t\norm{\xi}^{\alpha})d\xi \notag.
\end{align}

The first estimate of the heat kernel is given in the following :

\begin{lemma}
\label{heat-estimatesI} 
For any $t>0$, $\alpha >0$, and $x\in \Af$,  $Z(x,t)$ is well defined.
Furthermore, for any $t>0$ and $\alpha >0$,  the heat kernel $Z(x,t)$  satisfies the inequality
\[ \abs{Z(x,t)} \leq  C t^{-1/\alpha}, \qquad (x\in \Af), \]
where $C$ is a constant depending on $\alpha$.
\end{lemma}

\begin{proof}
Since the Haar measure of any ball is equal to its radius, we obtain
\begin{align*}
\abs{Z(x,t)} \leq \int_{\Af} \exp(-t\norm{\xi}^{\alpha}) d\xi <\int_0^{\infty} \exp(-ts^{\alpha}) ds &=t^{-1/\alpha}\Gamma(1/\alpha+1), 
\end{align*} where $\Gamma$ denotes the Archimedean gamma function.
Therefore,   $Z(x,t)$ is well defined and   the second assertion holds with $C=\Gamma(1/\alpha+1)$.
\end{proof}

\begin{proposition}\label{freeHeatKernel} 
The heat kernel $Z(x,t)$ is a positive function for all $x $ and $t>0 $. In addition 

\begin{align*}
Z(x,t)  =\sum_{\substack{n\in \Z \\ e^{\psi(n)} \leq \norm{x}^{-1}}} e^{\psi(n)} 
\left\{\exp(-t e^{\alpha\psi(n)}) - \exp(-t e^{\alpha\psi(n+1)}) \right\} \notag. 
\end{align*}

\end{proposition}

\begin{proof} 
From  Remark \ref{oscillatory_integral}, if $\norm{x}=e^{-\psi(m)}$, then
\begin{align*}
Z(x,t) &= \sum_{n=-\infty}^{\infty} \int_{S_n} \chi(-x \xi) \exp(-t\norm{\xi}^{\alpha}) d\xi \\
& = \sum_{n=-\infty}^{\infty} \exp(-t e^{\alpha\psi(n)}) \int_{S_n} \chi(-x \xi) d\xi \notag \\
& = \sum_{n=-\infty}^{m+1} \exp( -t e^{\alpha\psi(n)}) \int_{S_n} \chi(-x \xi) d\xi \notag \\ 
& = -\exp(-t e^{\alpha\psi(m+1)}) e^{\psi(m)} + 
\sum_{n=-\infty }^{m}\exp(-t e^{\alpha\psi(n)}) (e^{\psi(n)} - {e^{\psi(n-1)}})\\
& = \sum_{n=-\infty}^{m} 
e^{\psi(n)} \left\{\exp(-t e^{\alpha\psi(n)}) - \exp( -t e^{\alpha\psi(n+1)}) \right\} \notag \\
& =\sum_{\substack{n\in \Z \\ e^{\psi(n)} \leq \norm{x}^{-1}}} e^{\psi(n)} 
\left\{\exp(-t e^{\alpha\psi(n)}) - \exp(-t e^{\alpha\psi(n+1)}) \right\} \notag. 
\end{align*}

This implies that $Z(x,t)$ is a positive function for all $x $ and $t>0 $. 

\end{proof}

\begin{remark} It is important to notice that the expression of the heat kernel in Proposition \ref{freeHeatKernel}  does not depend on the algebraic structure of $\Af$. As a matter of fact, $Z(x,t)$ depends  only on the values  of the double sequence $(e^{\psi(n)})_{n \in \Z}$ and the ultrametric structure defined by this sequence on $\Af$. 
\end{remark}

\begin{corollary} 
\label{measure1}
The heat kernel is the distribution of a probability measure on $\Af$, i.e. $Z(x,t)\geq 0$ and
\[ \int_{\Af} Z(x,t)dx = 1, \] 
for all $t>0$. 
 
\end{corollary}

\begin{proof} From Proposition  \ref{freeHeatKernel} it follows that 
\begin{align*}
\int_{\Af} Z(x,t)dx &= \sum_{l=-\infty}^{\infty} \int_{S_l}  Z(x,t) dx \\ 
& = \sum_{l=-\infty}^{\infty}  Z(e^{\psi(l)} ,t) \big(e^{\psi(l)} - {e^{\psi(l-1)}}\big)\\
& =\sum_{l=-\infty}^{\infty} \left( \sum_{\substack{n\in \Z \\ e^{\psi(n)} \leq e^{-\psi(l)} }} e^{\psi(n)} 
\left\{\exp(-t e^{\alpha\psi(n)}) - \exp(-t e^{\alpha\psi(n+1)}) \right\} \right) \big(e^{\psi(l)} - {e^{\psi(l-1)}}\big)\\
& =\sum_{n=-\infty}^{\infty} \left( e^{\psi(n)} 
\left\{\exp(-t e^{\alpha\psi(n)}) - \exp(-t e^{\alpha\psi(n+1)}) \right\} \right) \left( \sum_{l=-\infty}^{-n}\big(e^{\psi(l)} - {e^{\psi(l-1)}}\big)\right)\\
&=\sum_{n=-\infty}^{\infty}  
\left\{\exp(-t e^{\alpha\psi(n)}) - \exp(-t e^{\alpha\psi(n+1)}) \right\}\\
&=1.
\end{align*}
\end{proof}

\begin{lemma}
\label{heat-estimatesII} 
For any $t>0$, $\alpha >0$ and $x\in \Af$, the heat kernel $Z(x,t)$ is positive and satisfies the inequality
\begin{align} Z(x,t) \leq \norm{x}^{-1} \big(1 - \exp(-t e^{\alpha\psi(m+1)})\big), \end{align}
where  $\norm{x}=e^{-\psi(m)}$.
\end{lemma}

\begin{proof} This follows from the inequality 
\begin{align*}
Z(x,t) &\leq \norm{x}^{-1}\sum_{\substack{n\in \Z \\ e^{\psi(n)} \leq \norm{x}^{-1}}} 
\left\{ \exp(-t e^{\alpha\psi(n)}) - \exp(-t e^{\alpha\psi(n+1)}) \right\}  \\
 &\leq \norm{x}^{-1}\big (1 - \exp(-t e^{\alpha\psi(m+1)})\big) 
\end{align*}
\end{proof}

\begin{proposition}
\label{heat-distribution}
The heat kernel satisfies the following properties:
\begin{itemize}
\item It is the distribution of a probability measure on $\Af$, i.e. $Z(x,t)\geq 0$ and
\[ \int_{\Af} Z(x,t)dx = 1, \] 
for all $t>0$. 
 
\item It converges to the Dirac distribution as $t$ tends to zero:
\[ \lim_{t \to 0} \int_{\Af} Z(x,t) f(x) dx = f(0), \] 
for all $f\in \D(\Af)$.
\item It has the Markovian property:
\[ Z(x,t+s) = \int_{\Af} Z(x-y,t) Z(y,s) dy. \]
\end{itemize}
\end{proposition}

\begin{proof} 
From Corollary \ref{measure1},  $Z(x,t)$ is in $L^1(\Af)$ for any $t>0$ and \[ \int_{\Af} Z(x,t) dx = 1. \] 
Using this equality, the fact that $f\in \D(\Af)$ is a locally constant function with  compact support and 
Lemma \ref{heat-estimatesII}, we conclude that
\[ \lim_{t \to 0} \int_{\Af} Z(x,t) \big(f(x)-f(0)\big) dx = 0. \]
The Markovian property follows from the Fourier inversion formula and the related property  of the exponential function.
\end{proof}

\begin{remark}It is worth mentioning that  the heat kernel associated to the isotropic Laplacian  of the ultrametric space $(\Af, d)$, with the Haar measure of $\Af$ as speed measure, and distribution function $e^{1/r}$, is equal to (see \cite{BGPW} for the definitions): 
$$
\sum_{\substack{n\in \Z \\ e^{\psi(n)} \leq \norm{x}^{-1}}} e^{\psi(n)} 
\left\{\exp(-t e^{\alpha\psi(n-1)}) - \exp(-t e^{\alpha\psi(n)}) \right\} \notag. 
$$
This kernel differes from  $Z(x,t)$ only by a term. However this one is very important when considering  bounds $\ref{heat-estimatesII}$.
\end{remark}

 \subsection[The  solution of the heat equation]{The  solution of the heat equation} 

Given $t> 0$  define the operator $T(t):L^2(\Af) \To L^2(\Af)$ by  the convolution with the heat kernel
\[T(t)f(x) = Z(x,t) * f(x), \qquad(f\in L^2(\Af)), \]  and let $T(0)$ be the identity operator on $L^2(\Af)$.
From Proposition \ref{heat-distribution} and Young's inequality   the family of operators $\{T(t)\}_{t \geq 0}$  is a strongly continuous contraction semigroup. 

The main theorem of the diffusion equation on the ring $\Af$  is the following.

\begin{theorem}\label{maintheorem} 
Let $\alpha>0$ and let $S(t)$ be  the $\Co_0$--semigroup generated by the operator $-\Da$. The operators $S(t)$ and  $T(t)$ agree for each
$t\geq 0$. In other words, for $f \in \Dom(\Da)$ and for $ t > 0$   the  solution of the abstract Cauchy problem (\ref{ACPinAf}) 
is given by $u(x,t)=Z(x,t) * f(x)$. 
\end{theorem}

\begin{proof} For $f\in L^1(\Af)\cap L^2(\Af) $,  the convolution  $u(x,t)=Z(x,t) * f(x)$ is in $L^1(\Af)\cap L^2(\Af)$ because $Z(x,t)$ is integrable for $t>0$. Then, the Fourier transform   $\F_{x\mapsto  \xi} u(x,t)$ is equal to 
$$
\hat{f}(\xi)\exp(-t\norm{\xi}^\alpha).
$$
From the commutative diagram above and the last equation, $S(t)(f)=T(t)(f)$.
Since  $L^1(\Af)\cap L^2(\Af)$ is dense in $L^2(\Af)$,
$S(t)(f)=T(t)$, for each $t\geq 0$. 
As a consequence, the function $u(x,t)=Z(x,t) * f(x)$ is a solution of the Cauchy  problem for any $f\in \Dom(\Da)$.
\end{proof}

\subsection[ Markov process on $\Af$]{ Markov process on $\Af$}
\label{markov_process}

In this section the fundamental solution of the heat equation,  $Z(x,t)$,  is shown to be the transition density function of a Markov 
process  on $\Af$. For general information on the theory of Markov process the reader can consult   \cite{Tai} and the classical writing \cite{Dyn}. 

 Let $\B$ denote  the Borel $\sigma$--algebra of $\Af$ and for $B \in \B$ write $1_B$ for the characterisctic or indicator function of $B$. Define 
\[ p(t,x,y) := Z(x-y,t) \quad (t>0,\, x,y\in \Af) \]
and
$$ P(t,x,B) =
\begin{cases}
\int_B p(t,x,y) dy & \text{if}\; t>0, x\in \Af, B\in \B,\\
1_B(x) & \text{if}\; t=0.
\end{cases}
$$

From Theorem \ref{heat-distribution},  it  follows that  $p(t,x,y)$ is a normal transition density and  $P(t,x,B)$ is a normal Markov
transition function on $\Af$ which corresponds to a Markov process on $\Af$ (see \cite[Section 2.1]{Dyn},  for further details).

In order to protray the properties of the path of the associated Markov process we first state the following:

\begin{lemma}\label{integralestimate} 
Let $k$ be any integer,  then
$$
\int_{\Af \backslash B_k}Z(x,t) \leq 1 - \exp(-t e^{\alpha\psi(-k)}).  $$
\end{lemma}

\begin{proof} Similar to Corollary  \ref{measure1}, we have the following 
\begin{align*}
\int_{\Af\backslash B_k} Z(x,t)dx &= \sum_{l=k+1}^{\infty} \int_{S_l}  Z(x,t) dx \\ 
& = \sum_{l=k+1}^{\infty}  Z(e^{\psi(l)} ,t) (e^{\psi(l)} - {e^{\psi(l-1)}})\\
& =\sum_{l=k+1}^{\infty} \left( \sum_{\substack{n\in \Z \\ e^{\psi(n)} \leq e^{-\psi(l)} }} e^{\psi(n)} 
\left\{\exp(-t e^{\alpha\psi(n)}) - \exp(-t e^{\alpha\psi(n+1)}) \right\} \right) (e^{\psi(l)} - {e^{\psi(l-1)}})\\
& =\sum_{n=-\infty}^{-(k+1)} \left( e^{\psi(n)} 
\left\{\exp(-t e^{\alpha\psi(n)}) - \exp(-t e^{\alpha\psi(n+1)}) \right\} \right) \left( \sum_{l=k+1}^{-n}\big(e^{\psi(l)} - {e^{\psi(l-1)}}\big)\right)\\
&\leq\sum_{n=-\infty}^{-(k+1)}  
\left\{\exp(-t e^{\alpha\psi(n)}) - \exp(-t e^{\alpha\psi(n+1)}) \right\}\\
&=1-\exp(-t e^{\alpha\psi(-k)}).
\end{align*}
\end{proof}

\begin{proposition}
\label{condition-L-M}
The transition function $P(t,y,B)$ satisfies the following two conditions:
\begin{enumerate}[a.]
\item For each $s\geq 0$ and  compact subset $B$ of $\Af$
\[ \lim_{x\to \infty} \sup_{t\leq s} P(t,x,B) = 0 \qquad  (\text{Condition LB} ). \]

\item For each $k>0$ and  compact subset $B$ of $\Af$
\[ \lim_{t\to 0+} \sup_{x\in B} P(t,x,\Af\setminus B_k(x)) = 0 \qquad  (\text{Condition MB} ). \]
\end{enumerate}
\end{proposition}

\begin{proof}
Let $d(x):=dist(x,B)=e^{\psi(-m_x)}$ where $m_x\in \Z$. From Lemma 
\ref{heat-estimatesII} it follows that
$$Z(x-y,t)\leq [d(x)]^{-1}\left(1-\exp(-s e^{\alpha \psi(m_x+1)})\right)$$ for any $y\in B$ and $t\leq s$. 
Since $B$ is compact and $\alpha$ is positive, $d(x)^{-1} \To 0$  and  $e^{\alpha \psi(m_x+1)}\To 0$, as $x\to \infty$. Hence
\[ P(t,x,B)\leq [d(x)]^{-1}\left(1-\exp(-s e^{\alpha \psi(m_x+1)})\right) \mu(B)\To 0 \] 
as $x\to \infty$. This implies Condition $L(B)$.

Presently, we establish Condition $M(B)$: for $y\in \Af \setminus B_k(x)$, we have  $\norm{x-y}> e^{\psi(k)}$. Therefore
$$
P(t,x,\Af \setminus B_k(x))=\int_{\Af \setminus B_k(x)} Z(x-y,t)dy  = \int_{\Af \setminus B_k(0)} Z(y,t)dy.
$$
From Lemma \ref{integralestimate}, 
$$
\int_{\Af \setminus B_k} Z(y,t)dy \leq 1-\exp(-te^{\alpha \psi(-k) } ) \To 0,\;  t\to 0^{+}.
$$ 
Since $P(t,x,B_k(x))$ is invariant under additive traslations, the last equation implies  
Condition $M(B)$.

\end{proof}

\begin{theorem} 
The heat kernel $Z(x,t)$ is the transition density of a time and space homogeneous Markov process which is bounded, right--continuous and has no discontinuities other than jumps.
\end{theorem}

\begin{proof}
The result follows from Proposition \ref{condition-L-M}  and the fact that $\Af$ is a second countable and locally compact  ultrametric space (see \cite[Theorem 3.6]{Dyn}). 
\end{proof}

\section[Cauchy problem for parabolic type equations on $\A$]{Cauchy problem for parabolic type equations on $\A$}
\label{ParabolicEquationsonA}

In this section an abstract Cauchy problem on $L^2(\A)$ is presented. First, we recollect several properties of the ring of ad\`eles 
$\A$. The abstract Cauchy problem on $L^2(\A)$ is studied by considering the fractional Laplacian   on the Archimedean completion, $\R$, and the pseudodifferential operator on $L^2(\Af)$, studied in the previous section. 

\subsection[The ring of ad\`eles $\A$]{The ring of ad\`eles $\A$}
\label{ring_adeles} 
In the present section, the ring of ad\`eles $\A$ of $\Q$ is described  as the product of its Archimedean part with its non--Archimedean component. We first consider  the locally compact and complete Archimedean field of real numbers $\R$. 

\subsubsection[The Archimedean place]{The Archimedean place} 
\label{analysis_R}
Recall that the real numbers $\R$ is the unique Archimedean completion of the rational numbers. As a locally compact Abelian group, $\R$, is autodual with pairing function given by $\chi_\infty(\xi_\infty x_\infty)$, where $\chi_{\infty}(x_{\infty})=e^{-2\pi i x_{\infty}}$ is the canonical character on $\R$. In addition, it is a commutative Lie group. The Schwartz space of $\R$, which we denote here by $\D(\R)$, consists of functions 
$\varphi_\infty : \R \To \C$ which are infinitely differentiable and rapidly decreasing. 
$\D(\R)$ has a countable family of seminorms which makes it a nuclear Fr\'echet space. Let $dx_{\infty}$ denote the usual Haar measure on $\R$. The Fourier transform 
$$
\F_{\infty}[\varphi_\infty](\xi_\infty) = 
\int_\R \varphi_\infty(x_\infty) \chi_\infty(\xi_\infty x_\infty) dx_\infty 
$$
is an isomorphism from $\D(\R)$ onto itself. Moreover, the Fourier inversion formula and the Parseval--Steklov identities hold on $\D(\R)$. Furthermore, $L^2(\R)$ is a separable Hilbert space, the Fourier transform is an isometry on $L^2(\R)$, and the Fourier  inversion formula and the Parseval--Steklov identity hold on $L^2(\R)$.

\begin{definition}
The \emph{\textsf{ad\`ele ring}} $\A$ of $\Q$ is defined as $\A = \R\times \Af$. 
\end{definition}

With the product topology, $\A$ is a second countable locally compact Abelian topological ring. If $\mu_\infty$ is the Haar measure on $\R$ and $\mu_{f}$ denotes the Haar measure on $\Af$, a Haar measure on $\A$ is given by the product measure 
$\mu = \mu_\infty \times \mu_{f}$.  Recall that, if $\chi_\infty$ and $\chi_f$ are the canonical characters on $\R$ and $\Af$, respectively, then $\chi = (\chi_\infty, \chi_f)$ defines a canonical character on $\A$. $\A$ is a selfdual group in the sense of Pontryagin and we have a paring $\chi_\infty(x_\infty\xi_\infty)$.

\subsubsection{Bruhat-Schwartz space} For any $\varphi_\infty \in \D(\R)$ and $\varphi_f \in \D(\Af)$, we have a function  $ \varphi$   on $\A$ given by  \[ \varphi(x) = \varphi_\infty(x_\infty)  \varphi_f(x_f) \]
for any ad\`ele $x=(x_\infty,x_f)$. These functions are  continuous on $\A$ and the  linear vector space generated by these functions is  linearly isomorphic to the algebraic tensor product $\D(\R) \otimes \D(\Af)$. In the following, we identify these spaces and write 
$\varphi=\varphi_\infty \otimes \varphi_f$.

Since  $\A$ is a locally compact  Abelian topological group  the  Bruhat--Schwartz space $\D(\A)$ has natural topology   described as follows. First, recall that $\D^l_{k}(\Af)$ denotes the set functions with  support on $B_k\subset \Af$  and parameter of constancy $l$. We have the  algebraic and topological tensor product  of a Fr\'echet space and finite dimensional space, given by
$$
\D(\R)\otimes\D_k^l(\Af)
$$
which represents a well defined class of functions on 
$\A$. These topological vector spaces are nuclear Fr\'echet, since $\D(\R)$ is nuclear Fr\'echet and $\D_k^l(\Af)$ has finite dimension. We have  
$$
\D(\A)= \varinjlim_{l \leq k} \D(\R)\otimes\D_k^l(\Af).
$$
The space of \emph{\textsf{Bruhat--Schwartz functions}} on $\A$ is the  algebraic  and topological tensor product of nuclear space  
vector  spaces $\D(\R)$ and $\D(\Af)$, i.e.
\[ \D(\A) = \D(\R) \otimes \D(\Af). \]

\subsubsection{The Fourier transform on $\A$}

The Fourier transform on $\D(\A)$ is defined  as
\[ \F[\varphi](\xi) = \int_{\A} \varphi(x) \chi(\xi x) dx, \]
for any $\xi\in \A$. 
It is well--defined on $\D(\A)$ and  for any function of the form 
$\varphi=\varphi_\infty \otimes \varphi_f$ it is given by
$$
\F[\varphi](\xi) = \F_\infty[\varphi_{\infty}](\xi_\infty)\otimes \F_{f}(\varphi_f)(\xi_f) \qquad (\xi=(\xi_\infty,\xi_f) \in \A)
$$ 
where $\F_{\infty}$ and $\F_f$  are the Fourier transforms on $\D(\R)$ and $\D(\Af)$, respectively. In other words, we have $\F_{\A} = \F_\R \otimes \F_{\Af}$.

The Fourier transform $\F:\D(\A)\To \D(\A)$ is a linear and continuous isomorphism. The inversion formula on $\D(\A)$ reads as
\[ \F^{-1}[\varphi](\xi) = \int_{\A} \widehat{\varphi}(-\xi) \chi(\xi x) d\xi, \qquad 
(\xi \in \A), \]
and  Parseval--Steklov equality as
$$
\int_{\A} \varphi(x) \overline{\psi(x)}dx = \int_{\A} \widehat{\varphi}(\xi) \overline{\widehat{\psi}(\xi)} d\xi.
$$ The space of square integrable functions $L^2(\A)$ on $\A$ is a separable Hilbert space since it is the Hilbert tensor product space $ L^2(\A) \cong L^2(\R) \otimes L^2(\Af)$. 
The Fourier transform $\F:L^2(\A)\To L^2(\A)$ is an isometry. The Fourier inversion formula and the Parseval-Steklov identity hold.

\subsection{A Cauchy Problem on  $L^2(\A)$} In this paragraph we consider a parabolic type equation on the complete ring of adeles.

\subsubsection{Archimedean heat kernel}
Let us recall the theory of the fractional heat kernel on the real line. For a complete review of this topic the reader may consult \cite{DGV} and the references therein.
For any $0 < \beta \leq 2$,  the fractional Laplatian $\Db_\infty:\Dom(A) \subset L^2(\R)\To L^2(\R)$ is given by $$ \Db_\infty \phi(x_{\infty})= \F^{-1}_{\xi_{\infty} \to x_{\infty}}\left[\abs{\xi}^\beta_{\infty} \F_{x_{\infty} \to \xi_{\infty}}[f]\right], $$
for any $\phi$ in the domain
$$
\Dom(\Db_\infty):=\left\{\, f \in L^2(\R) : \, \abs{\xi}_{\infty}^{\beta}\widehat{f} \in L^2(\R)  \,\right \}.
$$ 
Similar to the case of the finite ad\`ele ring, the operator $\Db_\R \phi(x_{\infty})$ is diagonalized by the unitary Fourier 
transform: if $\mb_{\infty}$ denotes the multiplicative operator on $L^2(\R)$ given by  
$f(\xi) \longmapsto  \abs{\xi}_{\infty}^\beta f(\xi)$, with domain $
\Dom(m^{\beta}_{\infty}):=\left\{\, f \in L^2(\R) : \, \abs{\xi}_{\infty}^{\beta}\widehat{f}(\xi) \in L^2(\R)  \,\right \} 
$, then    the following diagram commutes:
\begin{equation}\label{DiagramaConmutativo}
\begin{CD}
L^{2}(\R) @>\mathcal{F}>>L^{2}(\R) \\
@V{ \Db_{\infty}  }VV @VV{ m^\beta_{\infty}  }V\\
L^{2}(\R) @>\mathcal{F}>> L^{2}(\R). 
\end{CD}
\end{equation}

The pseudodifferential equation 
\begin{equation}
\label{HeatEquation}
\begin{cases}
\frac{{\partial}u(x_{\infty},t)}{{\partial t}} + \Db_{\infty} u(x_{\infty},t)=0,  \ x_{\infty} \in \R, \ t \geq 0;   \\
u(x,t)=f(x), \qquad  f \in \Dom(\Db_\infty) 
\end{cases}
\end{equation} is an abstract Cauchy problem whose  solution is given by the convolution of $f$ with the \emph{Archimedean heat kernel}:
$$
Z_{\infty}(x_\infty,t)=\int_{\R}\chi_{\infty}(\xi_{\infty} x_{\infty}) e^{-t\abs{\xi_{\infty}}^\beta}d\xi_{\infty}    \qquad (t >0).
$$
For $0 <\beta \leq 2$, the following bound holds
$$
\abs{Z_\infty(x_\infty,t)} \leq\frac{Ct^{1/\beta} }{t^{2/\beta}+x_\infty^2} \qquad( \text{for } t>0,\, x_\infty \in \R).
$$

Due to this bound, the Archimedean heat kernel satifies several properties: it is the distribution of a probability measure on $\R$; it converges to the Dirac delta distribution as $t$ tends to zero, and it satisfies the Markovian property.
Therefore the Archimedean heat kernel is the transition density of
a time and space homogeneous Markov process which is bounded, right--continuous and has no discontinuities other than jumps (see \cite[Section 2]{DGV}).

Moreover, the formula
$$S_{\infty}( f)=  f(x_\infty) \ast Z_{\infty}(x_\infty,t)  \qquad (f \in L^2(\R))$$  
defines a strongly continuous contraction
semigroup with  the unbounded operator  $\big(\Db_\infty, \Dom(\Db_\infty)\big)$ as infinitesimal generator.

Furthermore, there is a commutative diagram \begin{equation}\label{DiagramaConmutativoSemigrouposenA}
\begin{CD}
L^{2}(\R) @>\mathcal{F}>>L^{2}(\R) \\
@V{\exp(-t\Db_{\infty})  }VV @VV{ \exp(-t\mb_{\infty}) }V\\
L^{2}(\R) @>\mathcal{F}>> L^{2}(\R). 
\end{CD}
\end{equation}
where $\exp(-t\mb_{\infty})$ is the $C_0$--semigroup of contractions whose infinitesimal generator corresponds to the operator $-\mb_{\infty}$, under the Hille--Yoshida Theorem.

\subsubsection{Tensor product of  operators}
Let us briefly recall the definition of tensor product of operators on the Hilbert space $L^2(\mathbb{A})=L^2(\Af)\otimes L^2(\mathbb{R})$ (see \cite[Chapter VIII]{RS} for complete detail).

Given two (unbounded) closable 	 operators $(A, \Dom(A))$ and  $(B, \Dom(B))$ on  $L^2(\Af)$ and $L^2(\mathbb{R})$, respectively, the algebraic tensor product  $$\Dom(A) \otimes \Dom (B)=\left\{\, \sum_{\substack{\text{finite} }} \lambda_{i} \phi_f^i \otimes \phi_\infty^i \,:\, \phi_f^i \in \Dom(A), \, \phi_\infty^i \in \Dom(B) \,\right\} \subset L^2(\mathbb{A}) $$
 is dense in $L^2(\mathbb{A})$, and the operator  $A \otimes B$  given by 
 $$A \otimes B(\phi_f \otimes \phi_\infty )=A(\phi_f) \otimes B(\phi_\infty), $$
for $\phi_f \otimes \phi_\infty \in  \Dom(A) \otimes \Dom (A)$, is closable.

The \emph{tensor product} of $A$ and   $B$ is  the closure of the operator  $A \otimes B$ defined on the algebraic tensor product 
$\Dom(A) \otimes \Dom (B)$. We denote the closed operator by $A \otimes B$ and its domain by  $\Dom(A \otimes B)$.  Furthermore,
if $A$ and $B$ are selfadjoint, their tensor product $A\otimes B$ is essentially selfadjoint and the spectrum   $\sigma(A\otimes B)$ 
of $A\otimes B$ is the closure in $\C$ of $\sigma(A)\sigma(B)$, where $\sigma(A)$ and $\sigma(B)$ are the corresponding   spectrum of
$A$ and $B$.

On the other hand, if $A$ and $B$ are bounded operators, their tensor product $A\otimes B$ is bounded with operator norm 
$$\norm{A\otimes B}_{L^2(\A)}=\norm{A}_{L^2(\Af)} \norm{ B}_{L^2(\R)}.$$

Now, let us recall the definition of  the sum  of unbounded operators
on the Hilbert space $L^2(\A)=L^2(\Af)\otimes L^2(\mathbb{R})$ given by $A+B = A \otimes I + I \otimes B$. Once more, the algebraic tensor product $\Dom(A) \otimes \Dom(B ) \subset L^2(\A)$
is dense in $L^2 (\A)$ and the operator $A+B =A \otimes I + I\otimes B$ given by
$$
(A + B )(\phi_f \otimes \phi_\infty ) = A(\phi_f ) \otimes \phi_\infty + \phi_f\otimes B (\phi_\infty),$$
with $\phi_f \otimes \phi_\infty \in \Dom(A) \otimes \Dom(B)$ is essentially selfadjoint. The sum of $A$ and $B$ is the
closure of the operator $A + B$ defined on $Dom(A) \otimes Dom(B)$. We denote by $\Dom(A+B)$ the
domain of the this closed unbounded operator and with abuse of notation we denote this unbounded operator by  $A+B$. The
spectrum of $\sigma(A+B)$ of $A+B$ is the closure in $\C$ of $\sigma(A ) + \sigma(\Db )$, where $\sigma(A)$ and $\sigma(B)$ are the
corresponding spectrum of $A$ and $B$, respectively.

\subsubsection{Pseudodifferential operators on $\A$} 
First, notice that the multiplicative  operator   $\widehat{\mab}: L^2(\A) \To L^2(\A) $, given by $f(\xi) \longmapsto (\norm{\xi_f}^\alpha+\abs{\xi_\infty}^\beta) f(\xi)$,  with (dense) domain 
\begin{equation*}
\Dom(\widehat{\mab}):=\left\{\, f \in L^2(\A) \,: \,  \big(\norm{\xi_f}^\alpha+\abs{\xi_{\infty}}^\beta\big) \widehat{f}(\xi) \in L^2(\A)  \,\right \} 
\end{equation*} is selfadjoint and coincides with 
$\mab=\ma+\mb = \ma \otimes I + I \otimes \mb$ 
 on the set $\Dom(\ma ) \otimes \Dom(\mb ) \subset L^2(\A)$. Since $\mab$ is essentially selfadjoint on the domain $\Dom(\ma ) \otimes \Dom(\mb )$ it follows that  $\mab= \widehat{\mab}$.

For any $0<\alpha$ and $0 <\beta \leq 2$, consider the   pseudodifferential operator 
$\widehat{\Dab}:\Dom(\widehat{\Dab}) \subset L^2(\A)\To L^2(\A)$ defined by the formula
$$ \widehat{\Dab} \phi(x)= \F^{-1}_{\xi \to x}[\mab \F_{x \to \xi}[\phi]], $$
for any $\phi$ in the domain
\begin{equation*}
\Dom(\widehat{\Dab}):=\left\{\, f \in L^2(\A) : \, \mab(\widehat{\phi}) \in L^2(\A)  \,\right \}. 
\end{equation*} 
This unbounded operator is a positive selfadjoint operator which is  diagonalized by the (unitary) Fourier
transform $\F$, i.e.  the following diagram commutes:
\begin{equation}\label{DiagramaConmutativoenA}
\begin{CD}
L^{2}(\A) @>\mathcal{F}>>L^{2}(\A) \\
@V{\widehat{\Dab}  }VV @VV{ \mab=\widehat{\mab}  }V\\
L^{2}(\A) @>\mathcal{F}>> L^{2}(\A), 
\end{CD}
\end{equation}

Therefore, the operator  $\Dab=\Da+\Db = \Da \otimes I + I \otimes \Db$, which is essentially  selfadjoint over the domain  $\Dom(\Da ) \otimes \Dom(\Db ) \subset L^2(\A)$, is equal to the operator  $\widehat{\Dab}$.

\subsubsection{A heat equation on $\A$}
For  $f(x) \in \Dom(\Dab) \subset L^2(\A)$, consider  the abstract Cauchy problem
\begin{equation}
\label{ACPinA}
\begin{cases}
\frac{{\partial}u(x,t)}{{\partial t}} + \Dab u(x,t)=0,  \ x \in \A, \ t \geq 0   \\
\ u(x,t)=f(x).
\end{cases}
\end{equation} 
As mentioned above,  a function $u:\A\times [0,\infty) \To \C$ is called  a \textsf{solution} of the abstract Cauchy problem (\ref{ACPinA}) in the Hilbert space  $L^2(\A)$, if: 
\begin{enumerate}[a.]
\item $u:[0,\infty) \To L^2(\A)$ is a continuously differentiable function on the sense of Hilbert spaces,
\item $u(x,t)\in \Dom(\Dab)$, for all $t\geq 0$ and,
\item $u(x,t)$ is a solution of the initial value problem. 
\end{enumerate}
Furthermore,  this abstract Cauchy  problem is well posed and  its  solution is given by a strongly continuous contraction semigroup.  From the Hille--Yoshida theorem,  the  unbounded operator $-\Dab$ is the infinitesimal 
generator of a strongly continuous contraction semigroup $S_\A(t)=\exp(-t\Dab)$. 
Additionally, to the unbounded operator $-\mab$ there corresponds a strongly continuous contraction semigroup $\exp(-t\mab)$ with $\mab$ as infinitesimal generator.

From an argument as in \ref{semigroup_operators}, there is a commutative diagram:

\begin{equation}\label{DiagramaConmutativoSemigrouposenA}
\begin{CD}
L^{2}(\A) @>\mathcal{F}>>L^{2}(\A) \\
@V{\exp(-t\Dab)  }VV @VV{ \exp(-t\mab) }V\\
L^{2}(\A) @>\mathcal{F}>> L^{2}(\A). 
\end{CD}
\end{equation}

In order to describe the solution of problem (\ref{ACPinA}), for fixed $\alpha>0$ and $0<\beta \leq 2$, we define the \emph{adelic
heat kernel}  as 
$$
 Z_{\A}(x,t)=\int_{\A}\chi(-\xi x) e^{-t( \norm{\xi_f}^\alpha +\abs{\xi_{\infty}}^\beta)}d\xi    \qquad (t >0, \, x, \xi \in \A ),
$$ where  $\xi=(\xi_f,\xi_\infty)$. That is to say,
\begin{align}
Z_{\A}(x,t)&=\F^{-1}_{\A}(e^{-t( \norm{\xi_f}^\alpha +\abs{\xi_{\infty}}^\beta)}) \notag \\
&=\F^{-1}_{\infty}(e^{-t\abs{\xi_{\infty}}^\beta} )\F^{-1}_{f}(e^{-t\norm{\xi_{f}}^\alpha} ) \notag\\
&=Z_{f}(x_f,t) \otimes Z_{\infty}(x_\infty,t). \notag
\end{align}

\begin{proposition}
\label{heat-distribution-on-A}
The adelic  heat kernel, $Z_{\A}(x,t)$, satisfies the following properties:
\begin{itemize}
\item It is the distribution of a probability measure on $\Af$, i.e. $Z_\A(x,t)\geq 0$ and
\[ \int_{\A} Z_\A(x,t)dx = 1, \] 
for all $t>0$. 
 
\item It converges to the Dirac distribution as $t$ tends to zero:
\[ \lim_{t \to 0} \int_{\A} Z_{\A}(x,t) f(x) dx = f(0), \] 
for all $f\in \D(\A)$.
\item It has the Markovian property:
\[ Z(x,t+s) = \int_{\A} Z_{\A}(x-y,t) Z(y,s) dy. \]
\end{itemize}
\end{proposition}
\begin{proof}From the equality  $Z_{\A}(x,t)
=Z_{\Af}(x_f,t) \otimes Z_{\R}(x_\infty,t)$ it follows that
$Z_\A(x,t)$ is in $L^1(\A)$ for any $t>0$, and also \[ \int_{\A} Z_\A(x,t) dx = 1. \] 
Using the corresponding properties of the Arquimedian heat kernel and the finite adelic heat kernel, for $f\in \D(\A)$, we have
\[ \lim_{t \to 0} \int_{\A} Z_\A(x,t) \big(f(x)-f(0)\big) dx = 0. \]
The Markovian property follows from the Fourier inversion formula and the related property  of the exponential function.
\end{proof}

Now, for any $f\in L^2(\A)$, define
$$
T_\A(t)(f)(x)=\begin{cases} 
Z_{\A}(x,t)\ast f(x)    & t>0, \\
f(x) & t=0.
\end{cases}
$$ 
From Proposition \ref{heat-distribution-on-A} and Young's inequality it follows that $\{T_\A(t)\}_{t\geq 0}$ is a strongly continuous contraction semigroup.
On the other hand, from definition, it follows that
$$S_{\A}(t)(\phi_{f} \otimes \phi_{\infty})=\left( Z_{\Af}(x_{f},t)\ast \phi_{f} \right) \otimes \left( Z_{\R}(x_\infty,t)
\ast\phi_{\infty} \right). $$

 \begin{theorem}\label{solution_heatequation} If $f$ is any complex valued square integrable function on  $\Dom(\Dab)$, then the Cauchy problem
\begin{equation*}
\begin{cases}
\frac{{\partial}u(x,t)}{{\partial t}} + \Dab u(x,t) = 0,  \ x \in \A, \ t > 0,  \\  u(x,t) = f(x)
\end{cases}
\end{equation*} 
has a classical solution $u(x,t)$ determined by the convolution of $f$ with the
heat kernel $Z_{\A}(x,t)$. Moreover, $Z_{\A}(x,t)$ is the transition density  of a time and space homogeneous Markov process which is bounded, right-continuous and has no discontinuities other than jumps.
 \end{theorem}

\begin{proof} Similar to Section \ref{ParabolicEquationsonAf}, for 
$f \in L^1(\A) \cap L^2(\A)$, since the adelic heat kernel is absolute integral, the convolution $Z_{\A}(x,t)\ast f(x)$ is in $L^1(\A) \cap L^2(\A)$ and 
$$
\F_{x\to \xi }(Z_{\A}(x,t)\ast f(x))= \hat{f}(\xi)\exp(-t\norm{\xi_f}^\alpha+\abs{\xi_\infty}^\beta).
$$
Therefore $T_\A(t)=S_{\A}(t)$ coincides on a dense set of $L^2(\A)$. The properties of the Markov process follow because, the product of two Markov process which satisfy  conditions MB and LB also satisfies those conditions (see \cite[Section 4.9]{Zun}). 
\end{proof}

\begin{remark}A  slightly different proof of the Theorem above can be given as follows.
 Notice that the expression $\widetilde{S}_{\A}(t)=S_{\Af}(t)\otimes S_{\R}(t)$ gives a strongly continuous contraction  semigroup on $L^2(\A)$ which satisfies $
\norm{ \widetilde{S}_{\A}(t)} = \norm{S_{\Af}(t)}\norm{ S_{\mathbb{R}}(t)}
$. By the Leibniz rule, the infinitesimal generator of $\widetilde{S}_{\A}(t)$ is  $\Dab$. 
\end{remark}

\section*{Acknowledgements}

The  authors would like to thank Manuel Cruz-L\'opez, Sergii Torba and Wilson A. Zu\~niga--Galindo
for very useful discussions. Work was partially supported by CONACYT--FORDECYT  grant number 265667.

\end{document}